\documentclass[
12pt,
oneside,
a4paper,
DIV=calc,
abstract=true,
headsepline, 
]
{scrartcl}

\title{\LARGE Stability and rigidity of axisymmetric marginally outer trapped two-spheres}

\author{Gregory J. Galloway and Abraão Mendes}

\usepackage{lmodern}
\usepackage{amsmath,amsfonts,amsthm}

\allowdisplaybreaks

\usepackage{setspace}
\usepackage{amssymb}
\usepackage{color}
\usepackage[normalem]{ulem} 
\usepackage{mathtools}
\usepackage{lineno}
\usepackage{hyperref}

\newcounter{mnotecount}

\newcommand{\p}{\partial}
\newcommand{\X}{\mathfrak{X}}
\newcommand{\e}{\varepsilon}
\renewcommand{\d}{\mathrm{d}}
\renewcommand{\L}{\mathcal{L}}
\renewcommand{\H}{\mathcal{H}}
\renewcommand{\div}{\operatorname{div}}
\newcommand{\id}{\operatorname{id}}
\newcommand{\tr}{\operatorname{tr}}
\newcommand{\R}{\mathbb{R}}
\renewcommand{\th}{\theta}

\newtheorem{theorem}{Theorem}
\newtheorem{lemma}[theorem]{Lemma}
\newtheorem{proposition}[theorem]{Proposition}
\theoremstyle{remark}
\newtheorem{remark}[theorem]{Remark}

\usepackage{color}

\numberwithin{equation}{section}

\begin{document}

\maketitle

\begin{abstract}
In [7], H.~Bray, S.~Brendle, and A.~Neves studied rigidity properties of area-minimizing two-spheres in Riemannian three-manifolds with uniformly positive scalar curvature. In [13], these results were extended to marginally outer trapped surfaces (MOTS) in general initial data sets $(M^3,g,K)$ under a natural energy condition. In the present work, we refine the latter results to the setting of axisymmetric MOTS in initial data sets admitting a nontrivial Killing vector field. Conditions for the stability of such MOTS, as well as a new foliation lemma by axisymmetric surfaces of constant outward null expansion, are obtained. Finally, we discuss some aspects of the rotating Nariai spacetimes and their relation to these results.
\end{abstract}

\section{Introduction}

Minimal surfaces have long played a central role in differential geometry, arising as critical points of the area functional. In certain physical models, such as soap films spanning a wire frame, these surfaces describe equilibrium configurations determined by surface tension.

A rich theory has developed around the existence, stability, and rigidity of minimal surfaces. In particular, rigidity results show that under suitable curvature or symmetry assumptions, minimal surfaces are often uniquely determined. A striking example is the work of Bray, Brendle, and Neves \cite{BBN}, who established a sharp rigidity theorem for area-minimizing spheres in three-manifolds with positive scalar curvature. Their result highlights the profound connection between the existence of area-minimizing spheres that saturate a certain area upper bound and the global geometry of the ambient manifold.

In general relativity, marginally outer trapped surfaces (MOTS) naturally appear as Lorentzian analogues of minimal surfaces. Unlike minimal surfaces, MOTS are not defined variationally; rather, they are characterized by the vanishing of the outward null expansion $\theta^+$ along the orthogonal null geodesics at the surface. Nevertheless, MOTS possess a natural stability operator, closely analogous to the Jacobi operator for minimal surfaces. This operator governs the infinitesimal variations of $\theta^+$ and plays a central role in their analysis.

This analogy suggests that rigidity phenomena for MOTS may, in a sense, reflect those of minimal surfaces. Just as stable minimal spheres in certain ambient geometries must be round, one may ask: under what conditions is a MOTS uniquely determined by its surrounding spacetime geometry or initial data?

In this work, we address this question in the setting of rotationally symmetric spheres in suitable initial data sets. To some extent, the results presented here generalize those of \cite{GM2018} to the rotationally symmetric context.

\medskip

The paper is organized as follows. In Section~\ref{sec:preliminaries}, we present some basic definitions. In Section~\ref{sec:stability}, we address aspects of the stability of axisymmetric surfaces, and prove an infinitesimal rigidity statement (Proposition~\ref{prop:infinitesimal}) for closed, axisymmetric, stable, spherical MOTS in initial data sets admitting a Killing vector field; this result is subsequently used in the proof of the main theorem (Theorem~\ref{thm:main}). In Section~\ref{sec:foliation}, we establish, under natural conditions, the existence of a foliation by constant outward null expansion surfaces, each of which is invariant under a Killing vector field. Section~\ref{sec:main} contains the statement and proof of the main result of the paper. Finally, in Section~\ref{sec:Nariai}, we discuss certain aspects of the rotating Nariai spacetime and their relation to some of the results obtained in this work.

\newpage

\section{Preliminaries}\label{sec:preliminaries}

All manifolds in this paper are assumed to be orientable, unless otherwise stated.

Let $(M^3,g,K)$ be an initial data set, that is, $(M^3,g)$ is a three-dimensional Riemannian manifold and $K$ is a symmetric $(0,2)$-tensor defined on $M^3$.

Initial data sets arise naturally in general relativity as spacelike hypersurfaces $M^3$ in a spacetime $(\bar{M}^{3+1},\bar{g})$, where $g$ is the induced metric on $M^3$ and $K$ is its second fundamental form.

Let $\Sigma^2$ be a closed connected surface in $(M^3,g,K)$. As both $\Sigma^2$ and $M^3$ are orientable, there exists a unit vector field $\nu$ globally defined along $\Sigma^2$. If $\Sigma^2$ separates $M^3$, we denote by $M_+$ the connected component of $M \setminus \Sigma$ toward which $\nu$ points, and refer to it as the \emph{exterior} of $\Sigma^2$. In any case, by convention, we say that $\nu$ points to the outside of $\Sigma^2$.

The \emph{null expansion scalars} $\theta^+$ and $\theta^-$ of $\Sigma^2$ in $(M^3,g,K)$ are defined by
\begin{align*}
\theta^+ = \tr_\Sigma K + H, \quad 
\theta^- = \tr_\Sigma K - H,
\end{align*}
where $H = \div_\Sigma \nu$ is the mean curvature of $\Sigma^2$ in $(M^3,g)$, and $\tr_\Sigma K$ denotes the trace of the tensor $K$ restricted to tangent vectors to $\Sigma^2$.

The \emph{null second fundamental forms} $\chi^+$ and $\chi^-$ of $\Sigma^2$ in $(M^3,g,K)$ are given by
\begin{align*}
\chi^+=K|_\Sigma+A,\quad\chi^-=K|_\Sigma-A,
\end{align*} 
where $K|_\Sigma$ denotes the restriction of $K$ to the tangent spaces of $\Sigma^2$, and $A$ is the second fundamental form of $\Sigma^2$ in $(M^3,g)$. We adopt the sign convention in which $H = \tr A$, so that $\theta^\pm=\tr\chi^\pm$.

In a terminology inspired by the work of Penrose \cite{P}, a \emph{marginally outer trapped surface} (MOTS) is a closed surface $\Sigma^2$ whose outward null expansion vanishes identically, $\theta^+ \equiv 0$. Such surfaces represent the quasi-local boundary of a black hole in the initial data setting.

Let $\{\Sigma_t\}_{|t|<\epsilon}$ be a smooth variation of $\Sigma=\Sigma_0$, with variation vector field
\begin{align*}
\frac{\p}{\p t}\Big|_{t=0}=\phi\,\nu,
\end{align*}
for some smooth function $\phi$ on $\Sigma$. Let $\theta^+(t)$ denote the outward null expansion of $\Sigma_t$ computed with respect to the smooth choice of unit normals $\nu_t$ satisfying $\nu_0=\nu$.

It is well known (see \cite{AMS2005,AMS2008}; see also \cite{L}) that the first variation of $\theta^+$ is given by
\begin{align*}
\frac{\partial\theta^+}{\partial t}\Big|_{t=0}
= L\phi + \Big(\tau\theta^+ - \frac{1}{2}(\theta^+)^2\Big)\phi,
\end{align*}
where $\tau=\tr K$ is the trace of $K$ on $M^3$ with respect to $g$, and $L$ is a second-order elliptic operator acting on smooth functions $\phi:\Sigma^2\to\R$, defined by
\begin{align*}
L\phi = -\Delta\phi + 2\langle X,\nabla\phi\rangle + (Q - |X|^2 + \div X)\phi.
\end{align*}

Here $\Delta$ and $\nabla$ denote the Laplacian and gradient operators on $\Sigma^2$ with respect to the induced metric $\langle\,,\,\rangle$, and $X$ is the tangential vector field on $\Sigma^2$ metrically dual to the $1$-form $K(\nu,\cdot)|_\Sigma$.

The potential $Q$ is given by
\begin{align}\label{eq:Q}
Q = \kappa_\Sigma - (\mu + J(\nu)) - \frac{1}{2}|\chi^+|^2,
\end{align}
where $\kappa_\Sigma$ is the Gaussian curvature of $(\Sigma^2,\langle\,,\,\rangle)$, $\mu$ and $J$ are the energy and momentum densities associated to the initial data,
\begin{align*}
\mu=\frac{1}{2}(R-|K|^2+\tau^2),\quad J=\div(K-\tau g),
\end{align*}
and $R$ is the scalar curvature of $(M^3,g)$.

When $\Sigma^2$ is a MOTS, the first variation of $\theta^+$ reduces to the operator $L$, referred to as the \emph{MOTS stability operator}. In the time-symmetric case, that is, when $K \equiv 0$, the outward null expansion $\theta^+$ reduces to the mean curvature of $\Sigma^2$, and a MOTS is precisely a minimal surface. In this situation, $L$ coincides with the minimal surface stability operator, namely the Jacobi operator.

As discussed by Andersson, Mars, and Simon (\cite[Appendix~B]{AMS2008}), although the operator $L$ is not symmetric, it possesses a real eigenvalue $\lambda_1=\lambda_1(L)$ such that $\operatorname{Re}\lambda \geq \lambda_1$ for all (possibly complex) eigenvalues $\lambda$. Moreover, the eigenspace associated with $\lambda_1$ is one-dimensional and contains a smooth positive eigenfunction $u$. The eigenvalue $\lambda_1$ is called the \emph{principal eigenvalue}, and $u$ is called a \emph{principal eigenfunction} of $L$.

Then, as in \cite{AMS2005,AMS2008}, a MOTS $\Sigma^2$ is said to be \emph{stable} if $\lambda_1 \ge 0$. This is equivalent to the existence of a positive smooth function $\phi$ on $\Sigma^2$ such that $L \phi \ge 0$.

Another important differential operator on $\Sigma^2$ is the \emph{formal adjoint} of $L$:
\begin{align*}
L^*\phi=-\Delta\phi-2\langle X,\nabla\phi\rangle+(Q-|X|^2-\div X)\phi.
\end{align*}

A direct integration by parts, together with the divergence theorem, yields
\begin{align*}
\int_\Sigma\psi L\phi=\int_\Sigma\phi L^*\psi.
\end{align*}
Moreover, $L^*$ has the same principal eigenvalue as $L$. Indeed, if $u>0$ and $u^*>0$ are principal eigenfunctions of $L$ and $L^*$, respectively, then
\begin{align*}
(\lambda_1(L)-\lambda_1(L^*))\int_\Sigma u u^*=\int_\Sigma u^*Lu-\int_\Sigma uL^*u^*=0.
\end{align*} 
Since $u$ and $u^*$ are positive, this implies $\lambda_1(L)=\lambda_1(L^*)$ (see \cite[Appendix~B]{AMS2008}).

\section{Stability and infinitesimal rigidity of axisymmetric MOTS}\label{sec:stability}

Inspired by the works of Anderson, Mars, and Simon \cite{AMS2005,AMS2008}, Jaramillo, Reiris, and Dain \cite{JRD} introduced the notion of the \emph{spacetime stably outermost condition}\linebreak for axisymmetric MOTS $\Sigma^2$, i.e.\ MOTS that are invariant under an axial Killing vector $\eta$ on $\Sigma^2$. While the setting in \cite{JRD} is somewhat more general, here we restrict the discussion to initial data sets.

In the case of initial data sets, the `spacetime stably outermost condition' introduced in \cite{JRD} can be restated as follows: Assuming certain quantities on $\Sigma^2$ are axisymmetric, $\Sigma^2$ is said to be 
\emph{axisymmetrically stable} if there exists a positive 
\emph{axisymmetric function}
$\phi$ on $\Sigma^2$ such that $L \phi \ge 0$.

In particular, axisymmetrically stable MOTS are also stable in the standard sense. The first result we prove in this paper establishes a converse.

\begin{lemma}\label{lemma:1}
Let $(\Sigma^n,g_\Sigma)$ be a closed Riemannian manifold, and let $\eta$ be a Killing vector field on $\Sigma^n$. If $X \in \X(\Sigma)$ and $Q \in C^\infty(\Sigma)$ are invariant under~$\eta$, then the principal eigenfunction of the operator
\begin{align*}
L u = -\Delta u + 2 \langle X, \nabla u \rangle + Q u
\end{align*}
is also invariant under $\eta$.
\end{lemma}

\begin{proof}
Let $\psi_t$ be the flow generated by $\eta$. Saying that $X$ and $Q$ are invariant under $\eta$ means, in terms of the flow, that
$$
\d \psi_t \cdot X = X \circ \psi_t, \quad Q \circ \psi_t = Q.
$$

Let $u > 0$ be a principal eigenfunction of $L$, i.e. 
$$
L u = \lambda_1 u,
$$
where $\lambda_1 = \lambda_1(L)$ is the principal eigenvalue of $L$. For simplicity, fix $t$ and write $R = \psi_t$. Since $R$ is an isometry, it follows that
$$
\Delta (u \circ R) = (\Delta u) \circ R.
$$

On the other hand,
$$
\begin{aligned}
\langle X, \nabla (u \circ R) \rangle 
&= \d u \cdot (\d R \cdot X) = \d u \cdot (X \circ R) \\
&= (\d u \cdot X) \circ R = \langle X, \nabla u \rangle \circ R.
\end{aligned}
$$
Therefore,
$$
\begin{aligned}
L(u \circ R) 
&= -\Delta(u \circ R) + 2 \langle X, \nabla (u \circ R) \rangle + Q (u \circ R) \\
&= -(\Delta u) \circ R + 2 \langle X, \nabla u \rangle \circ R + (Q u) \circ R \\
&= (L u) \circ R = \lambda_1 (u \circ R).
\end{aligned}
$$

This shows that $u \circ R$ is an eigenfunction associated with the principal eigenvalue $\lambda_1$. Since $\lambda_1$ is simple, there exists a constant $c(R)$ such that
$$
u \circ R = c(R) u.
$$

In other words,
$$
u \circ \psi_t = c(t) u,
$$
where $c(t) := c(\psi_t)$. Because $\psi_0 = \id$, it follows that
$$
c(0) = 1.
$$
Furthermore, the group property of the flow implies
$$
c(t + s) = c(t) \cdot c(s).
$$
Hence, there exists a constant $a \in \mathbb{R}$ such that
$$
c(t) = e^{a t}.
$$

Suppose, for the sake of contradiction, that $a > 0$. Then
$$
c(t) = e^{a t} \to +\infty \quad \text{as} \quad t \to +\infty,
$$
which contradicts the boundedness of $u \circ \psi_t$, since $u$ is continuous on the compact manifold $\Sigma^n$. Similarly, if $a < 0$, then
$$
c(t) = e^{a t} \to +\infty \quad \text{as} \quad t \to -\infty,
$$
contradicting the boundedness of $u \circ \psi_t$ in backward time.

Therefore, it must be that
$$
a = 0,
$$
and consequently,
$$
u \circ \psi_t = u,
$$
i.e.\ $u$ is invariant under the flow generated by $\eta$.
\end{proof}

The next result provides natural conditions under which the vector field $X\in\X(\Sigma)$, dual to $K(\nu,\cdot)|_\Sigma$, and the potential $Q$ from \eqref{eq:Q} are invariant under a Killing vector field $\eta$.

\begin{lemma}\label{lemma:2}
Let $(M^3,g,K)$ be an initial data set, and let $\eta$ be a Killing vector field on $(M^3,g)$. 
Consider a closed connected surface $\Sigma^2\subset M^3$. Suppose that both $K$ and $\Sigma^2$ are invariant under the flow generated by $\eta$. 
Then the vector field $X\in\X(\Sigma)$, associated with $K(\nu,\cdot)|_\Sigma$, as well as the potential
\begin{align*}
Q = \kappa_\Sigma - (\mu+J(\nu)) - \frac{1}{2}|\chi^+|^2
\end{align*}
are invariant under $\eta$.
\end{lemma}

\begin{proof}
Let $\psi_t$ denote the flow generated by $\eta$. Since $\psi_t$ is an isometry of $(M^3,g)$ and $\Sigma^2$ is invariant under $\eta$, it follows that $\psi_t$ maps $\Sigma^2$ onto itself and pushes forward the unit normal $\nu$ to another unit normal along $\Sigma^2$:
$$
\d\psi_t \cdot \nu_p \quad \text{is a unit normal vector at} \quad \psi_t(p).
$$

Because $\Sigma^2$ is connected and two-sided, the unit normal vector field $\nu$ is uniquely defined up to sign. Moreover, since $\d\psi_0 \cdot \nu = \nu$, the continuity of $\psi_t$ and the connectedness of $\Sigma^2$ ensure that this sign cannot change along the flow. Hence,
$$
\d\psi_t \cdot \nu = \nu \circ \psi_t,
$$
showing that $\nu$ is invariant under $\eta$.

Let $\omega = K(\nu, \cdot)$ be the $1$-form on $\Sigma^2$ obtained by contracting $K$ with $\nu$. Then
$$
\L_\eta \omega = (\L_\eta K)(\nu, \cdot) + K(\L_\eta \nu, \cdot) = 0,
$$
since both $K$ and $\nu$ are invariant under $\eta$.

On the other hand, for any tangential vector field $Y$ on $\Sigma^2$, we have the identity
$$
(\L_\eta \omega)(Y) = g(\L_\eta X, Y) + (\L_\eta g)(X, Y).
$$
Since $\eta$ is Killing, we have $\L_\eta g = 0$ and thus
$$
g(\L_\eta X, Y) = (\L_\eta \omega)(Y) = 0.
$$
Because this holds for all $Y$, we deduce that $\L_\eta X = 0$, i.e.\ $X$ is invariant under~$\eta$.

Furthermore, $\psi_t|_\Sigma$ is an isometry of $(\Sigma^2,\langle\,,\,\rangle)$, so $\kappa_\Sigma = \kappa_\Sigma \circ \psi_t$ on $\Sigma^2$; in other words, $\kappa_\Sigma$ is invariant under $\eta$.

To see that
$$
\mu = \frac{1}{2} \big( R - |K|^2 + (\tr K)^2 \big)
$$
is invariant under $\eta$, note that:
\begin{itemize}
    \item $R$ is invariant under $\eta$ because $\psi_t$ is an isometry;
    \item $|K|^2$ is invariant under $\eta$, since
    $$
    \L_\eta(|K|^2)
    = 2\langle \L_\eta K, K\rangle
      - 2\langle \L_\eta g, K\circ K\rangle
    = 0,
    $$
    where $K\circ K$ is defined by
    $$
    (K\circ K)_{ij} = g^{kl} K_{ik} K_{jl},
    $$
    and we used that both $K$ and $g$ are invariant under $\eta$;
    \item $\tr K$ is invariant under $\eta$ because
    $$
    \L_\eta(\tr K) = \tr(\L_\eta K) - \langle \L_\eta g, K\rangle = 0.
    $$
\end{itemize}

Similarly, we can prove that all remaining terms in $Q$ are invariant under the flow generated by $\eta$.
\end{proof}

Let $\Sigma^2$ be a MOTS in $(M^3,g,K)$ and consider the ``symmetrized'' operator on $\Sigma^2$,
\begin{align*}
\L u = -\Delta u + Q u,
\end{align*}
which is associated with the MOTS stability operator $L$, as studied by the first-named author and Schoen in \cite{GS}. By a simple modification of the arguments presented in the proof of the main result in \cite{GS}, 
one deduces that $\lambda_1(\L) \geq \lambda_1(L)$ (see \cite[Lemma~2.2]{GalMOTS1}).
In particular, if $\Sigma^2$ is stable, then $\lambda_1(\L) \geq 0$. Since $\L$ is symmetric, this implies that
\begin{align*}
\int_\Sigma( |\nabla f|^2 + Q f^2) \geq 0
\end{align*}
for every smooth function $f$ on $\Sigma^2$.

In \cite{JRD} (and, again, restricting to initial data sets), a strengthened stability inequality is obtained 
for surfaces $\Sigma^2$ that are axisymmetrically stable as described above, for axisymmetric functions $f$ (again, assuming certain other quantities along 
$\Sigma^2$ are axisymmetric). In view of Lemmas \ref{lemma:1} and \ref{lemma:2}, this stability inequality still holds assuming 
$\Sigma^2$ is stable in the standard sense, provided $K$ and $\Sigma^2$ are invariant under a nontrivial Killing vector field $\eta$ on $(M^3,g)$.

\begin{proposition}\label{prop:stability}
Let $(M^3,g,K)$ be an initial data set, and let $\eta$ be a nontrivial Killing vector field on 
$(M^3,g)$. Assume that $K$ is invariant under $\eta$. If $\Sigma^2$ is a closed stable MOTS in $(M^3,g,K)$ that is invariant under $\eta$, then
\begin{align*}
\int_{\Sigma\setminus\{\eta=0\}}|X^\eta|^2 f^2\le\int_\Sigma(|\nabla f|^2 + Q f^2),
\end{align*}
for every axisymmetric smooth function $f$ on $\Sigma^2$, where $X^\eta$ is the projection of $X$ onto $\eta$,
\begin{align}\label{eq:xnu}
X^\eta = \frac{\langle X,\eta\rangle}{\langle \eta,\eta\rangle}\,\eta.
\end{align}
\end{proposition}

Once one invokes Lemmas~\ref{lemma:1} and~\ref{lemma:2}, the proposition essentially follows from \cite[Lemma~1]{JRD}. However, because of the substantial differences in notation, together with certain small additional issues we address, we find it useful to include the proof here.  While the proofs are somewhat similar, our proof roughly follows along the lines of arguments in the main result in \cite{GS}.

\begin{proof} 
Let $u>0$ be a principal eigenfunction of $L$. Since $\Sigma^2$ is stable, we have $\lambda_1\ge0$, so
\begin{align*}
0
&\le \lambda_1 u
= -\Delta u + 2\langle X,\nabla u\rangle
+ (Q + \div X - |X|^2)u \\
&= -\Delta u + 2\langle X^\perp, \nabla u\rangle
+ (Q + \div X^\perp - |X^\perp|^2) u
+ (\div X^\eta - |X^\eta|^2) u
\end{align*}
on $\Sigma\setminus\{\eta=0\}$, where $X^\perp = X - X^\eta$.
Here we used that $\langle X^\eta,\nabla u\rangle = 0$, since by Lemmas~\ref{lemma:1} and~\ref{lemma:2} the function $u$ is invariant under $\eta$ (i.e.\ $u$ is axisymmetric).

Thus, on $\Sigma\setminus\{\eta=0\}$,
\begin{align}\label{eq:aux.1}
0 \le \lambda_1
= \div Y - |Y|^2 + Q + \div X^\eta - |X^\eta|^2,
\end{align}
where $Y = X^\perp - \nabla\ln u$.

Multiplying \eqref{eq:aux.1} by $f^2$ yields
\begin{align*}
0 &\le f^2 \div Y - f^2 |Y|^2 + Q f^2
+ f^2 \div X^\eta - f^2 |X^\eta|^2 \\
&= \div (f^2 Y) - 2f \langle \nabla f, Y\rangle - f^2|Y|^2 + Q f^2  + \div (f^2 X^\eta)\\
&\quad\quad - 2f \langle \nabla f, X^\eta\rangle - f^2|X^\eta|^2.
\end{align*}

Since $f$ is axisymmetric, $\langle\nabla f, X^\eta\rangle = 0$. Moreover,
\begin{align*}
-2f\langle\nabla f,Y\rangle - f^2 |Y|^2
= |\nabla f|^2 - |\nabla f + f Y|^2 \le |\nabla f|^2.
\end{align*}
Therefore,
\begin{align*}
|X^\eta|^2 f^2  \le |\nabla f|^2 + Q f^2
+ \div (f^2(Y+X^\eta)).
\end{align*}

\medskip

It is well known that the only closed orientable surfaces admitting a nontrivial Killing vector field are the two-sphere $S^2$ and the two-torus $T^2$.

\paragraph{Case 1: $\Sigma^2$ is a two-torus.}  
In this case, $\eta \neq 0$ everywhere on $\Sigma^2$. Integrating the last inequality over $\Sigma^2$ and applying the divergence theorem immediately yields the desired estimate.

\paragraph{Case 2: $\Sigma^2$ is a two-sphere.}  
Here, $\eta$ has exactly two zeros, say $p_1$ and $p_2$. For small $\varepsilon > 0$, let $D_\varepsilon(p_i)$ denote the geodesic disk of radius $\varepsilon$ centered at $p_i$. Integrating the inequality over
$$
\Sigma_\varepsilon := \Sigma \setminus( D_\varepsilon(p_1) \cup D_\varepsilon(p_2))
$$
and applying the divergence theorem gives
\begin{align*}
\int_{\Sigma_\varepsilon} |X^\eta|^2 f^2
&\le \int_{\Sigma_\varepsilon} ( |\nabla f|^2 + Q f^2)
+ \sum_{i=1}^2 \int_{\partial D_\varepsilon(p_i)}
( \langle Y, N_i \rangle + \langle X^\eta, N_i \rangle ) f^2,
\end{align*}
where $N_i$ is the unit normal to $\partial D_\varepsilon(p_i)$, tangent to $\Sigma^2$ and pointing into $D_\varepsilon(p_i)$.

Note that $|X^\perp|^2+|X^\eta|^2=|X|^2$. Since $X$ and $u > 0$ are smooth, both $X^\eta$ and $Y$ are uniformly bounded on $\partial D_\varepsilon(p_i)$, $i = 1,2$. Therefore, letting $\varepsilon \searrow 0$ yields
\begin{align*}
\int_{\Sigma\setminus\{p_1,p_2\}} |X^\eta|^2 f^2
\le \int_\Sigma(|\nabla f|^2 + Q f^2),
\end{align*}
which is the desired inequality.
\end{proof}

\begin{remark}
In the context of initial data sets, and under appropriate axisymmetry assumptions, certain results known to hold for MOTS that are axisymmetrically stable (e.g.\ because their proofs rely on the stability inequality) will now hold for MOTS that are stable in the standard sense (in particular, for MOTS that are locally weakly outermost).
\end{remark}

As a first application of Proposition \ref{prop:stability}, we obtain an upper bound on the area of a stable MOTS that is invariant under a Killing vector field $\eta$, and establish certain rigidity if the area bound is saturated. This result is a refinement in the axisymmetric setting of the area bound (and associated rigidity) obtained in Proposition 3.1 in \cite{GM2018}, where now the area bound involves the angular momentum. 

\begin{proposition}[Infinitesimal rigidity]\label{prop:infinitesimal}
Let $(M^3,g,K)$ be an initial data set, and let $\eta$ be a nontrivial Killing vector field on $(M^3,g)$. Assume that $K$ is invariant under $\eta$. If $\Sigma^2$ is a closed stable MOTS in $(M^3,g,K)$ that is invariant under $\eta$, and if 
$$
\mu + J(\nu) \ge c
$$
along $\Sigma^2$ for some constant $c > 0$, then $\Sigma^2$ is topologically $S^2$ and its area satisfies
\begin{align}\label{eq:area_inequality}
|\Sigma| \le \frac{4\pi}{c+\omega},
\end{align}
where
\begin{align}\label{eq:omega}
\omega := \frac{1}{|\Sigma|} \int_{\Sigma \setminus \{\eta=0\}} |X^\eta|^2 .
\end{align}
Moreover, if equality holds in \eqref{eq:area_inequality}, then:
\begin{enumerate}
\item[\em (1)] The null second fundamental form $\chi^+$ of $\Sigma^2$ vanishes.
\item[\em (2)]The Gaussian curvature of $\Sigma^2$ satisfies
$$
\kappa_\Sigma = c + |X^\eta|^2 \quad \text{along} \quad \Sigma\setminus\{\eta=0\};
$$
in particular, $|X^\eta|^2$ extends smoothly to the entire surface $\Sigma^2$.
\item[\em (3)] One has $\mu+J(\nu) \equiv c$ along $\Sigma^2$.
\item[\em (4)] The principal eigenvalue $\lambda_1(L)$ equals zero.
\end{enumerate}
\end{proposition}

\begin{proof}
By Proposition~\ref{prop:stability}, the bilinear form
$$
q(f,h) := \int_{\Sigma \setminus \{\eta=0\}} ( \langle\nabla f,\nabla h\rangle + (Q - |X^\eta|^2) fh )
$$
is positive semidefinite on the space of axisymmetric smooth functions on $\Sigma^2$.

Choosing $f=h \equiv 1$ gives
\begin{align*}
0 &\le q(1,1)
   = \int_{\Sigma \setminus \{\eta=0\}} (Q - |X^\eta|^2) \\
  &= \int_\Sigma \Big( \kappa_\Sigma - (\mu+J(\nu)) - \frac12 |\chi^+|^2 \Big) 
     - \omega|\Sigma| \\
  &\le 2\pi \chi(\Sigma) - (c+\omega)|\Sigma|,
\end{align*}
where we have used $\mu+J(\nu) \ge c$ and the Gauss-Bonnet theorem.

Since $c>0$ and $\omega \ge 0$, we deduce $\chi(\Sigma) > 0$; hence $\Sigma$ is topologically $S^2$, and
$$
|\Sigma| \le \frac{4\pi}{c+\omega}.
$$

Moreover, if equality holds in \eqref{eq:area_inequality}, then all inequalities above are equalities. In particular, $\chi^+ \equiv 0$, $\mu+J(\nu) \equiv c$, and $q(1,1) = 0$. 

For any $\alpha \in \mathbb{R}$ and any axisymmetric function $h$ on $\Sigma^2$,
$$
0 \le q(\alpha+h,\alpha+h) 
   = \alpha^2 q(1,1) + 2\alpha\,q(1,h) + q(h,h) 
   = 2\alpha\,q(1,h) + q(h,h),
$$
which implies $q(1,h) = 0$, that is,
\begin{align*}
\int_{\Sigma\setminus\{\eta=0\}}(Q-|X^\eta|^2)h=0
\end{align*}
for every axisymmetric $h$. Since $Q - |X^\eta|^2$ is axisymmetric, this yields
$$
Q - |X^\eta|^2 \equiv 0 \quad \text{on} \quad \Sigma\setminus\{\eta=0\};
$$
thus
$$
|X^\eta|^2 = Q = \kappa_\Sigma - (\mu+J(\nu)) - \frac12 |\chi^+|^2 = \kappa_\Sigma - c.
$$

Finally, integrating inequality \eqref{eq:aux.1} over 
$\Sigma_\varepsilon := \Sigma \setminus (D_\varepsilon(p_1) \cup D_\varepsilon(p_2))$, we find
\begin{align*}
0 \le \lambda_1 |\Sigma_\varepsilon| 
&\le \int_{\Sigma_\varepsilon} (Q - |X^\eta|^2 - |Y|^2)
   + \sum_{i=1}^2 \int_{\partial D_\varepsilon(p_i)} 
       (\langle Y, N_i\rangle + \langle X^\eta, N_i\rangle) \\
&\le \sum_{i=1}^2 \int_{\partial D_\varepsilon(p_i)} 
       (\langle Y, N_i\rangle + \langle X^\eta, N_i\rangle)\rightarrow0\quad\text{as}\quad\varepsilon\searrow0.
\end{align*}
Hence $\lambda_1 = 0$, as claimed.
\end{proof}

\begin{remark} 
In our notation, the Komar angular momentum $\mathcal{J}$ (see, e.g., \cite{DR}) of a surface 
$\Sigma^2$ invariant under a Killing vector field $\eta$ is obtained  by integrating 
$K(\nu, \eta) = \langle X, \eta\rangle$ over $\Sigma$. Hence, we see from \eqref{eq:xnu} and the expression for $\omega$ that, if $\mathcal{J} \ne 0$,  then $\omega > 0$ (equivalently, if  
$\omega = 0$ then $\mathcal{J} =0$).  Thus, while the angular momentum determines a lower bound for the area (see \cite{DR, JRD}), it also influences the upper bound.  In Section~\ref{sec:Nariai}, we explore this influence on the class of rotating Nariai spacetimes, which arise as a certain limit of the Kerr-de Sitter spacetime.
\end{remark}

\section{Foliation lemma}\label{sec:foliation}

An important tool in proving splitting results in differential geometry is the existence of foliations by constant mean curvature surfaces in the Riemannian setting, or by constant null expansion surfaces in the initial data context (see, e.g., \cite{ACG,BBN,GM2018,GM2025,MM,M}).

In this section, we establish, under natural assumptions, the existence of a foliation by constant outward null expansion surfaces, each of which is invariant under a Killing vector field~$\eta$.  We begin with two preliminary lemmas.

Let $\Sigma^2$ be a topological two-sphere endowed with a Riemannian metric $g_\Sigma$, and let $\eta$ be a nontrivial Killing vector field on $(\Sigma^2, g_\Sigma)$. As a consequence, $(\Sigma^2, g_\Sigma)$ is rotationally symmetric (see \cite{CLT}), meaning that the metric $g_\Sigma$ can be expressed in the form
\begin{align}\label{eq:coord_sys}
g_\Sigma = d\theta^2 + \rho^2(\theta)\,d\phi^2, \quad \phi \sim \phi + 2\pi,
\end{align}
where $\rho > 0$ on $(0, \pi)$ and satisfies $\rho(0) = \rho(\pi) = 0$, and $\eta=\p_\phi$.

It is well known that a necessary and sufficient condition for $\rho$ to be smooth on $S^2$ is that (see \cite{Petersen})
$$
\rho'(0) = 1, \quad \rho'(\pi) = -1,
$$
and that all its even-order derivatives vanish at the poles, i.e.
$$
\rho^{(2k)}(0) = \rho^{(2k)}(\pi) = 0 \quad \text{for all} \quad k = 1, 2, \ldots
$$

Consider the linear subspace
$$
F = \{ f \in C^{2,\alpha}(\Sigma); \L_\eta f = 0 \}
$$
of axisymmetric functions in $E = C^{2,\alpha}(\Sigma)$, and let $P$ denote the projection of $E$ onto $F$.
If $T$ is the period of $\psi_t$, the flow generated by $\eta$, then
$$
(Pf)(x) = \frac{1}{T} \int_0^T f(\psi_t(x)) \, dt.
$$
Clearly, $(Pf) \circ \psi_s = Pf$ for every $s$, that is, $Pf$ is invariant under $\eta$.

In the coordinate system \eqref{eq:coord_sys}, the operator $P$ takes the form
$$
(Pf)(\theta,\phi) = \frac{1}{2\pi} \int_0^{2\pi} f(\theta, \phi + t) \, dt
$$
and satisfies
$$
(Pf)(\theta,\phi) = (Pf)(\theta,0), \quad \forall (\theta,\phi).
$$

\medskip

\begin{lemma}
If $X \in \mathfrak{X}(\Sigma)$ and $Q \in C^\infty(\Sigma)$ are invariant under $\eta$,
then the projection $P$ commutes with the operator
$$
L f = -\Delta f + 2\langle X, \nabla f \rangle + Q f, \quad f\in C^\infty(\Sigma).
$$
\end{lemma}

\begin{proof}
Applying $L$ to $Pf$ gives
$$
L(Pf) = \frac{1}{T} \int_0^T L(f\circ\psi_t) \, dt.
$$

Since $X$ and $Q$ are invariant under $\eta$, and $\eta$ is a Killing vector field on $(\Sigma^2, g_\Sigma)$, it follows from the proof of Lemma~\ref{lemma:1} that $L$ is preserved by the flow~$\psi_t$:
$$
L(f\circ\psi_t) = (Lf)\circ\psi_t.
$$
Therefore,
$$
L(Pf) = \frac{1}{T} \int_0^T (Lf)\circ\psi_t \, dt = P(Lf),
$$
which shows that $P$ and $L$ commute.
\end{proof}

\begin{lemma}\label{lemma:4}
Let $(M^{n+1},g)$ be a complete Riemannian manifold, and let $\eta$ be a nontrivial Killing vector field on $(M^{n+1},g)$. Consider a connected, two-sided, embedded hypersurface $\Sigma^n$ in $(M^{n+1},g)$. Assume that $\Sigma^n$ is invariant under $\eta$. Then the following hold:
\begin{enumerate}
\item[\em (1)] The second fundamental form $A$ and the mean curvature $H = \tr A$ of $\Sigma^n$ in $(M^{n+1},g)$ are invariant under $\eta$. In particular, if $K$ is a symmetric $(0,2)$-tensor on $(M^{n+1},g)$ that is invariant under $\eta$, then the null second\linebreak fundamental form $\chi^+ := K|_\Sigma + A$ and the outward null expansion\linebreak $\theta^+ := \tr\chi^+$ of $\Sigma^n$ in $(M^{n+1},g,K)$ are also invariant under $\eta$.
\item[\em (2)] Let $f \in C^\infty(\Sigma)$, and define
$$
\Sigma_f = \{ \exp_p(f(p)\,\nu_p) ; p \in \Sigma \},
$$
where $\nu$ is a globally defined unit normal vector field along $\Sigma^n$. If $f$ is invariant under $\eta$, then $\Sigma_f$ is also invariant under $\eta$.
\end{enumerate}
\end{lemma}

\begin{remark}
If $(M^{n+1},g)$ is not necessarily complete, then item~(2) of Lemma~\ref{lemma:4} still holds for $f \in C^\infty(\Sigma)$ provided that $\|f\|_\infty$ is sufficiently small. Item~(1), on the other hand, holds regardless of the completeness of $(M^{n+1},g)$.
\end{remark}

\begin{proof}
Let $\psi_t$ denote the flow generated by $\eta$. Since $R=\psi_t$ is an isometry of $(M^{n+1},g)$, the Levi-Civita connection is invariant under $R$, i.e.
$$
\nabla_{R_* X} R_* Y = R_* (\nabla_X Y).
$$

From the proof of Lemma~\ref{lemma:2}, we also know that
$$
R_* \nu = \nu \circ R.
$$
Hence,
\begin{align*}
(R^* A)(X,Y) &= A(R_* X, R_* Y) = \langle \nabla_{R_* X} (\nu \circ R), R_* Y \rangle \\
&= \langle \nabla_{R_* X} R_* \nu, R_* Y \rangle 
= \langle R_* (\nabla_X \nu), R_* Y \rangle \\
&= \langle \nabla_X \nu, Y \rangle = A(X,Y),
\end{align*}
which shows that $A$ is invariant under $\psi_t$ for every $t$; equivalently, $\L_\eta A = 0$.

Since both $A$ and the induced metric $g_\Sigma$ are invariant under $\eta$, the mean curvature $H = \tr A = \tr_{g_\Sigma} A$ is also invariant (see the proof of Lemma~\ref{lemma:2}). Using the assumed invariance of $K$,
the invariance of $\chi^+$ and $\theta^+$ follows similarly. 

Now let $f \in C^\infty(\Sigma)$ be invariant under $\eta$, and fix $R = \psi_t$ for some $t$. Since $R$ is an isometry, the curve
$$
\gamma(s) := R(\exp_p(s\,\nu_p))
$$
is a geodesic with initial conditions $\gamma(0) = R(p)$ and $\gamma'(0) = \d R \cdot \nu_p$. Therefore,
$$
\gamma(s) = \exp_{R(p)}(s \, \d R \cdot \nu_p), \quad \forall s.
$$
In particular,
$$
R(\exp_p(f(p)\, \nu_p)) = \exp_{R(p)}(f(p)\, \d R \cdot \nu_p).
$$

If we set $q = R(p)$, then $\d R \cdot \nu_p = \nu_q$, and since $f$ is invariant under $R$, $f(p) = f(q)$. Thus,
$$
R(\exp_p(f(p)\, \nu_p)) = \exp_q(f(q)\, \nu_q) \in \Sigma_f.
$$

This shows that $R(\Sigma_f) \subset \Sigma_f$, and hence $\Sigma_f$ is invariant under the flow $\psi_t$ for every $t$.
\end{proof}

We now present the foliation lemma.

\begin{lemma}[Foliation lemma]\label{lemma:foliation}
Let $(M^3,g,K)$ be an initial data set, and let $\eta$ be a nontrivial Killing vector field on $(M^3,g)$. Assume that $K$ is invariant under $\eta$. If $\Sigma^2$ is a closed axisymmetric MOTS in $(M^3,g,K)$ with $\lambda_1(L)=0$ that is homeomorphic to $S^2$, then there exists a neighborhood $U \cong (-\varepsilon,\varepsilon) \times \Sigma$ of $\Sigma \cong \{0\} \times \Sigma$ in $M$ such that:
\begin{enumerate}
\item[\em (1)] The metric $g$ has the orthogonal decomposition:
$$
g = \phi^2dt^2 + g_t \quad \text{on} \quad U
$$
for some axisymmetric positive function $\phi:U\to\R$, where $g_t$ is the induced metric on $\Sigma_t \cong \{t\} \times \Sigma$.
\item[\em (2)] Each $\Sigma_t$ is an axisymmetric surface in $(M^3,g,K)$ with constant outward null expansion $\theta^+(t)$ with respect to the outward unit normal $\nu_t = \phi^{-1} \partial_t$, where $\nu_0 = \nu$. Furthermore, $\Sigma_t$ is contained in the exterior of $\Sigma$ for each $t\in[0,\varepsilon)$.
\end{enumerate}
\end{lemma}

\begin{remark}
Although the result is stated for two-dimensional spherical MOTS, it remains valid for  MOTS 
$\Sigma^n$ of arbitrary dimension, assuming that $\eta$ generates an $S^1$-action on 
$\Sigma^n$.
\end{remark}

\begin{proof}
Fix $0<\alpha<1$ and define
$$
B_\delta(0)=\{f \in C^{2,\alpha}(\Sigma); \|f\|_{2,\alpha} < \delta \}.
$$

Given $f \in B_\delta(0)$, consider
$$
\Sigma_f = \{\exp_p(f(p)\,\nu_p); p \in \Sigma\}
$$
and let $\theta_f^+$ denote the outward null expansion of $\Sigma_f$ with respect to the (suitably chosen) outward unit normal $\nu_f$ to $\Sigma_f$.
By taking $\delta>0$ smaller if necessary, we may assume that $\Sigma_f$ is embedded in $M$ for every $f \in B_\delta(0)$.

Define
$$
\Theta : B_\delta(0) \times \mathbb{R} \to P(C^{0,\alpha}(\Sigma)) \times \mathbb{R} \times P^\perp(C^{2,\alpha}(\Sigma))
$$
by
$$
\Theta(f,k) = \Big( P\theta_f^+ - k, \int_\Sigma Pf, P^\perp f \Big),
$$
where
$$
P^\perp f = f - Pf,
$$
and where $P$ is the projection operator introduced above. Note that $P^2 = P$, and hence $P(P^\perp f) = P^\perp(P f)=0$. 

The derivative of $\Theta$ at $(0,0)$, in the direction of $(f,k) \in C^{2,\alpha}(\Sigma) \times \mathbb{R}$, is
$$
\d\Theta_{(0,0)} \cdot (f,k)
= \frac{d}{ds}\Big|_{s=0} \Theta(sf, sk)
= \Big( P(Lf) - k, \int_\Sigma Pf, P^\perp f \Big).
$$

We claim that $\d\Theta_{(0,0)}$ is an isomorphism.

\paragraph{Injectivity.}
Suppose $\d\Theta_{(0,0)} \cdot (f,k) = (0,0,0)$.  
Since $P$ and $L$ commute, we have
$$
L(Pf) = k, \quad \int_\Sigma Pf = 0, \quad f = Pf.
$$
Observe that
$$
k \int_\Sigma u^*
= \int_\Sigma u^* L(Pf)
= \int_\Sigma (Pf)L^*u^*
= 0,
$$
which implies $k=0$. Here $u^*$ is a principal eigenfunction of the formal adjoint $L^*$ associated with the principal eigenvalue $\lambda_1(L^*)=\lambda_1(L)=0$. Therefore, $Pf = c\,u$ for some $c \in \mathbb{R}$, since $\lambda_1(L)=0$ is simple, where $u>0$ is a principal eigenfunction of $L$. 
The condition $\int_\Sigma Pf = 0$ then forces $c=0$, so that $f = Pf = 0$.  
Thus $\d\Theta_{(0,0)}$ is injective.

\paragraph{Surjectivity.}
Let $(v,c,w)$ be an arbitrary element of
$$
P(C^{0,\alpha}(\Sigma)) \times \mathbb{R} \times P^\perp(C^{2,\alpha}(\Sigma)).
$$
Choose $k_0 \in \mathbb{R}$ such that
$$
\int_\Sigma (v + k_0) u^* = 0.
$$
By the Fredholm alternative, there exists $f_0 \in C^{2,\alpha}(\Sigma)$ such that
$$
L f_0 = v+k_0.
$$
Since $P$ and $L$ commute and $v \in P(C^{2,\alpha}(\Sigma))$, we have
$$
L(Pf_0) - k_0 = P v = v.
$$

On the other hand, because $u$ is invariant under $\eta$, we can choose $s_0 \in \mathbb{R}$ such that
$$
\int_\Sigma P(f_0 + s_0 u)
= \int_\Sigma (Pf_0 + s_0 u)
= c.
$$

Finally, since $P w = 0$, it is not difficult to see that the pair $$(Pf_0 + s_0 u + w, k_0)\in C^{2,\alpha}(\Sigma)\times\R$$ satisfies
$$
\d\Theta_{(0,0)} \cdot (Pf_0 + s_0 u + w, k_0) = (v, c, w).
$$
Thus $\d\Theta_{(0,0)}$ is surjective.

\medskip

Then, by the inverse function theorem, there exists a differentiable path
$$
(f(s),k(s)) \in B_\delta(0) \times \mathbb{R}, \quad |s| < \varepsilon,
$$
with $(f(0),k(0)) = (0,0)$, such that
$$
\Theta(f(s),k(s)) = (0,s,0),
$$
that is,
$$
P\theta_{f(s)}^+ = k(s), \quad \int_\Sigma P(f(s)) = s, \quad P(f(s)) = f(s),\quad\forall s.
$$

Since $f(s)$ is invariant under $\eta$, by Lemma~\ref{lemma:4} the same holds for the graph
$\Sigma_s := \Sigma_{f(s)}$ and its outward null expansion
$\theta^+(s) := \theta_{f(s)}^+$. Thus,
$$
\theta^+(s) = k(s), \quad \int_\Sigma f(s) = s.
$$

Similar to the injectivity argument, $L(f'(0)) = k'(0)$ implies $f'(0) = c\,u$ for some $c \in \mathbb{R}$. Since $\int_\Sigma f'(0) = 1$, we have $c > 0$. Therefore, by taking $\varepsilon > 0$ smaller if necessary, the family $\{\Sigma_s\}_{|s| < \varepsilon}$ forms a foliation of a neighborhood $U\cong(-\varepsilon,\varepsilon)\times\Sigma$ of $\Sigma_0 = \Sigma$ by constant outward null expansion surfaces $\Sigma_s\cong\{s\}\times\Sigma$, all of them invariant under $\eta$. Furthermore, $\Sigma_s$ is contained in $M_+$ for $s\in[0,\varepsilon)$.

Moreover, one can introduce coordinates so that, up to isometry,
\begin{align}\label{eq:orthog}
U = (-\varepsilon,\varepsilon) \times \Sigma, \qquad g = \phi^2 dt^2 + g_t \quad \text{on} \quad U,
\end{align}
where $g_t$ is the induced metric on $\Sigma_t = \{t\} \times \Sigma$ and $\phi$ is invariant under $\eta$.

Very briefly, this can be accomplished as follows. Let 
$\tau: U \to (-\varepsilon, \varepsilon)$ be a defining function for the foliation: $\Sigma_s = \{\tau = s\}$. Introduce the vector field 
$$
X = \frac{\nabla\tau}{|\nabla\tau|^2},
$$
which is orthogonal to the leaves $\Sigma_s$, and let $\Phi_t$ be the flow generated by $X$.

Given a coordinate chart $(V,x^1,x^2)$ on $\Sigma^2$, consider the map
$$
\Psi: (-\varepsilon, \varepsilon) \times V \to M^3, \quad 
\Psi(t, x^1, x^2) = \Phi_t(x^1,x^2).
$$
Note that $\d\Psi \cdot \frac{\partial}{\partial t} = X$, so that it is orthogonal to the leaves. 
Further, using that $\d\tau \cdot X = 1$, one sees that the flow $\Phi_t$ preserves the leaf structure 
-- i.e.\ shifts leaves to leaves.

From this, it follows that $\d\Psi \cdot \frac{\partial}{\partial x^i}$, $i= 1,2$, is tangent to the leaves. Hence $\d\Psi \cdot \frac{\partial}{\partial t}$ and $\d\Psi \cdot \frac{\partial}{\partial x^i}$ are orthogonal. By pulling back the metric via $\Psi$, we obtain the desired structure \eqref{eq:orthog} 
with respect to the coordinate neighborhood~$V$. This can be globalized by observing that the coordinate $t$ has an invariant meaning.

Finally, with the metric in the form \eqref{eq:orthog}, using 
$\L_\eta t = 0$ and basic properties of Lie derivatives, we obtain
$$
0 = \L_\eta g = (\L_\eta\phi^2)\, dt \otimes dt + \L_\eta g_t 
= 2\phi\, (\L_\eta\phi)\, dt \otimes dt,
$$
which shows that $\phi$ is invariant under $\eta$.
\end{proof}

\section{Main result: Local splitting}\label{sec:main}

In this section, we state and prove the main result of this work. Before doing so, we recall some important definitions. 

A closed surface $\Sigma'$ in an initial data set $(M^3,g,K)$ is said to be \emph{outer trapped} if $\theta^+<0$ on $\Sigma'$. A separating MOTS $\Sigma$ is \emph{weakly outermost} if there exists no outer trapped surface $\Sigma'$ that is homologous to $\Sigma$ and lies entirely in the exterior region $M_+$ of $\Sigma$. $\Sigma$ is said to be \emph{locally weakly outermost} if there exists a neighborhood $U \subset M$, $\Sigma \subset U$, such that $\Sigma$ is weakly outermost in $U$. Finally, $\Sigma$ is said to be \emph{outer area-minimizing} if $|\Sigma|\le|\Sigma'|$ for every surface $\Sigma'$ in $M_+$ that is homologous to $\Sigma$.

\begin{theorem}\label{thm:main}
Let $(M^3,g,K)$ be an initial data set satisfying $\mu-|J|\ge c$ for some positive constant $c$. Let $\eta$ be a nontrivial Killing vector field on $(M^3,g)$, and assume that $K$ is invariant under $\eta$. Consider an axisymmetric, closed, weakly outermost MOTS $\Sigma^2$ in $(M^3,g,K)$. Then $\Sigma^2$ is topologically a two-sphere and its area satisfies
\begin{align}\label{eq:main}
|\Sigma|\le\frac{4\pi}{c+\omega},
\end{align}
where
\begin{align*}
\omega:=\frac{1}{|\Sigma|}\int_{\Sigma\setminus\{\eta=0\}}|X^\eta|^2
\end{align*}
and 
\begin{align*}
X^\eta:=\frac{\langle X,\eta\rangle}{\langle\eta,\eta\rangle}\,\eta
\end{align*}
is the projection of $X$ onto $\eta$.

If equality holds in \eqref{eq:main} and, in addition, $\Sigma^2$ is outer area-minimizing and minimizes $\omega$, then the following hold:

\begin{enumerate}
\item[\em (1)]  There exists an outer neighborhood $U$ of $\Sigma$ such that, up to isometry, 
\begin{align*}
g=dt^2+g_\Sigma \quad \text{on} \quad U = [0,\epsilon) \times \Sigma,
\end{align*}
where $g_\Sigma$ is the round metric on $\Sigma^2$ of constant Gaussian curvature\linebreak $\kappa_\Sigma=c$, 
\item[\em (2)] $K = \alpha\, dt \otimes dt$
for some function $\alpha = \alpha(t)$,
\item[\em (3)]  $\mu = c$ and $J = 0$ on $U$, and 
\item[\em (4)]  for each $t \in [0,\epsilon)$, $\Sigma_t = \{t\} \times \Sigma$ is invariant under $\eta$, $\chi_t^{\pm} = 0$, $\omega_t = 0$ (hence $\Sigma_t$ has zero angular momentum), and $|\Sigma_t| = 4\pi/c$.
\end{enumerate}
\end{theorem}

\begin{lemma}\label{lemma:beta}
Consider the metric
$$
g = dt^2 + d\theta^2 + \rho^2(\theta)\, d\phi^2 \quad \text{on} \quad M^3 = \mathbb{R} \times S^2,
$$
with Killing vector field $\eta = \partial_\phi$, and let $K$ be a symmetric $(0,2)$-tensor on $M^3$ of the form
$$
K = \alpha(t)\, dt \otimes dt + \beta(\theta)\,(dt \otimes d\phi + d\phi \otimes dt).
$$
If the slice $\Sigma_0 = \{0\}\times S^2$ minimizes $\omega$, then $\beta \equiv 0$.
\end{lemma}

\begin{proof}
Let $\Sigma_f$ denote the graph $t = f(\theta)$. We aim to estimate the quantity
$$
\omega(\Sigma_f)=\frac{1}{|\Sigma_f|} \int_{\Sigma_f} |X^{\partial_\phi}|^2.
$$

First, observe that the unit normal to $\Sigma_f$ is
$$
\nu_f = \frac{\partial_t - f'(\theta)\partial_\theta}{\sqrt{1 + f'(\theta)^2}}.
$$
Hence,
$$
|X^{\partial_\phi}|^2 = \frac{K(\nu_f, \partial_\phi)^2}{\langle \partial_\phi, \partial_\phi \rangle} 
= \frac{\beta^2(\theta)}{\rho^2(\theta)(1 + f'(\theta)^2)}.
$$

On the other hand, the area element of $\Sigma_f$ is
$$
dA_f = \sqrt{1 + f'(\theta)^2}\,dA_0,
$$
where $dA_0 = \rho(\theta)\, d\theta\, d\phi$ is the area element of the slice $\Sigma_0 : t = 0$. Therefore,
$$
\begin{aligned}
\int_{\Sigma_f} |X^{\partial_\phi}|^2 
&= \int_{\Sigma_0} \frac{\beta^2(\theta)}{\rho^2(\theta) (1 + f'(\theta)^2)} \sqrt{1 + f'(\theta)^2}\, dA_0 \\
&= \int_{\Sigma_0} \frac{\beta^2(\theta)}{\rho^2(\theta)  \sqrt{1 + f'(\theta)^2}}\, dA_0 \\
&\le \int_{\Sigma_0} \frac{\beta^2(\theta)}{\rho^2(\theta)}\, dA_0 
= \int_{\Sigma_0} |X^{\partial_\phi}|^2.
\end{aligned}
$$

In particular, if $\beta \not\equiv 0$, one can choose a graph $\Sigma_f$ such that
$$
\int_{\Sigma_f} |X^{\partial_\phi}|^2 < \int_{\Sigma_0} |X^{\partial_\phi}|^2,
$$
and thus
$$\omega(\Sigma_f)=\frac{1}{|\Sigma_f|}\int_{\Sigma_f} |X^{\partial_\phi}|^2<\frac{1}{|\Sigma_0|}\int_{\Sigma_0} |X^{\partial_\phi}|^2=\omega(\Sigma_0),$$
since $|\Sigma_0|\le|\Sigma_f|$ for every $f=f(\theta)$.
\end{proof}

\begin{proof}[Proof of Theorem \ref{thm:main}]
Since $\Sigma$ is weakly outermost -- and hence stable -- and $\mu+J(\nu)\ge\mu-|J|\ge c$ along $\Sigma$, it follows from Proposition~\ref{prop:infinitesimal} that
$$
|\Sigma|\le\frac{4\pi}{c+\omega}.
$$
Moreover, if equality holds, then $\chi^+$ vanishes, $\mu+J(\nu)=\mu-|J|\equiv c$, and $\kappa_\Sigma=c+|X^\eta|^2$ along $\Sigma$. Also, the principal eigenvalue $\lambda_1(L)$ of $L$ equals zero. 

Therefore, by the Foliation lemma (Lemma~\ref{lemma:foliation}), there exists an outer neighborhood $U\cong[0,\varepsilon)\times\Sigma$ of $\Sigma\cong\{0\}\times\Sigma$ in $M$ such that
$$
g=\phi^2dt^2+g_t\quad\text{on}\quad U,
$$
and each $\Sigma_t\cong\{t\}\times\Sigma$ is an axisymmetric surface in $(M,g,K)$ with constant outward null expansion $\theta^+(t)$ with respect to the unit normal $\nu_t=\phi^{-1}\partial_t$. Furthermore, the function $\phi:U\to\mathbb{R}$ is axisymmetric.

The first variation of $\theta(t)=\theta^+(t)$ is given by
$$
\theta'=-\Delta\phi+2\langle X,\nabla\phi\rangle+\Big(Q+\div X-|X|^2+\theta\tau-\frac{1}{2}\theta^2\Big)\phi.
$$
Thus,
\begin{align}
\frac{\theta'}{\phi}-\theta\tau
&=-\frac{\Delta\phi}{\phi}+2\langle X^\perp+X^\eta,\nabla\ln\phi\rangle+Q \nonumber\\
&\quad\quad +\div X^\perp+\div X^\eta-|X^\perp|^2-|X^\eta|^2-\frac{1}{2}\theta^2 \nonumber\\
&=-|Y|^2+\div Y+Q+\div X^\eta-|X^\eta|^2-\frac{1}{2}\theta^2\label{eq:aux4}\\
&\le Q+\div(Y+X^\eta)-|X^\eta|^2,\nonumber
\end{align}
where $Y=X^\perp-\nabla\ln\phi$. Above we used that $\langle X^\eta,\nabla\ln\phi\rangle=0$, since $\phi$ is axisymmetric on $\Sigma_t$. Recall that
$$
X^\eta:=\frac{\langle X,\eta\rangle}{\langle\eta,\eta\rangle}\,\eta
$$
is the projection of $X$ onto $\eta$ and
$$
X^\perp:=X-X^\eta.
$$

Using the same argument as in the proof of Proposition~\ref{prop:stability}, we obtain
\begin{align*}
\theta'\int_{\Sigma_t}\frac{1}{\phi}-\theta\int_{\Sigma_t}\tau
&\le\int_{\Sigma_t}Q-\int_{\Sigma_t\setminus\{\eta=0\}}|X^\eta|^2 \\
&=\int_{\Sigma_t}\Big(\kappa_{\Sigma_t}-(\mu+J(\nu_t))-\frac{1}{2}|\chi_t^+|^2\Big)-\omega_t|\Sigma_t| \\
&\le\int_{\Sigma_t}(\kappa_{\Sigma_t}-c)-\omega_t|\Sigma_t| \\
&=4\pi-(c+\omega_t)|\Sigma_t|,
\end{align*}
where
$$
\omega_t:=\frac{1}{|\Sigma_t|}\int_{\Sigma_t\setminus\{\eta=0\}}|X^\eta|^2.
$$
Above we used that $\mu+J(\nu_t)\ge\mu-|J|\ge c$ along $\Sigma_t$ for each $t$.

Because we are assuming that $(c+\omega)|\Sigma|=4\pi$ and that $\Sigma$ is an outer area-minimizing surface (in particular, $|\Sigma|\le|\Sigma_t|$), we have
$$
\theta'\int_{\Sigma_t}\frac{1}{\phi}-\theta\int_{\Sigma_t}\tau
\le(c+\omega)|\Sigma|-(c+\omega_t)|\Sigma_t|
\le(\omega-\omega_t)|\Sigma|.
$$

Also, because $\omega\le\omega_t$ for every $t\in[0,\varepsilon)$, since we are assuming that $\Sigma$ minimizes $\omega$, we have
\begin{align*}
\theta'\int_{\Sigma_t}\frac{1}{\phi}-\theta\int_{\Sigma_t}\tau\le0,
\end{align*}
which is equivalent to
\begin{align*}
\Big(\theta(t)e^{-\int_0^t\xi(s)\,ds}\Big)'\le0,
\end{align*}
where
\begin{align*}
\xi(t):=\frac{\int_{\Sigma_t}\tau}{\int_{\Sigma_t}\frac{1}{\phi}}.
\end{align*}

Therefore,
\begin{align*}
\theta(t)e^{-\int_0^t\xi(s)\,ds}\le\theta(0)=0,
\end{align*}
that is, $\theta(t)\le0$ for every $t\in[0,\varepsilon)$. Since $\theta(t)$ is constant on $\Sigma_t$, and $\Sigma$ is weakly outermost, we conclude that $\theta(t)=0$ for each $t$, i.e.\ each $\Sigma_t$ is a MOTS. This gives that all inequalities above must be equalities. In particular, for each~$t$, we have:
\begin{itemize}
\item $|\Sigma_t|=|\Sigma|$ and $\omega_t=\omega$; as a corollary,
\begin{align}\label{eq:aux3}
|\Sigma_t|=\frac{4\pi}{c+\omega_t}.
\end{align}
\item $\chi_t^+\equiv0$ and $\mu+J(\nu_t)=\mu-|J|\equiv c$ on $\Sigma_t$.
\item $Y=X^\perp-\nabla\ln\phi\equiv0$ along $\Sigma_t$.
\end{itemize}

Since $\Sigma$ is weakly outermost, and $\Sigma_t\subset M_+$ is homologous to $\Sigma$, we obtain that $\Sigma_t$ is also weakly outermost -- and hence stable. Thus, as a consequence of Proposition~\ref{prop:infinitesimal} and equation \eqref{eq:aux3}, we have $\kappa_{\Sigma_t}=c+|X^\eta|^2$ along $\Sigma_t$. Therefore, it follows from \eqref{eq:aux4} that, on $\Sigma_t$,
\begin{align*}
\div X^\eta&=|X^\eta|^2-Q=|X^\eta|^2-\kappa_{\Sigma_t}+(\mu+J(\nu_t))+\frac{1}{2}|\chi_t^+|^2\\
&=-c+c=0.
\end{align*}
Moreover, since $\Sigma$ is outer area-minimizing and $|\Sigma_t|=|\Sigma|$, we obtain that $\Sigma_t$ is locally area-minimizing for each $t\in(0,\varepsilon)$. In particular, $\Sigma_t$ is minimal. Because $\theta^+(t)=\theta(t)\equiv0$, we conclude that $\tr_{\Sigma_t}K\equiv0$ for each $t$. 

This gives that each $\Sigma_t$ is a \textit{marginally inner trapped surface} (MITS): $\theta^-\equiv0$. Then, taking the first variation of $\theta^-(t)$, with $\phi^-=-\phi$, we have:
\begin{align*}
0=\frac{d\theta^-}{dt}&=-\Delta\phi^-+2\langle X^-,\nabla\phi^-\rangle+(Q^--|X^-|^2+\div X^-)\phi^-,
\end{align*}
where $X^-$ is the vector field dual to the $1$-form $K(-\nu_t,\cdot)=-K(\nu_t,\cdot)$ on $\Sigma_t$, that is, $X^-=-X$, and
\begin{align*}
Q^-=\kappa_{\Sigma_t}-(\mu+J(-\nu_t))-\frac{1}{2}|\chi_t^-|^2=|X^\eta|^2-2|J|-\frac{1}{2}|\chi_t^-|^2,
\end{align*}
where above we used that $\kappa_{\Sigma_t}=c+|X^\eta|^2$, $\mu=c+|J|$, and $J(\nu_t)=-|J|$ on~$\Sigma_t$. Therefore,
\begin{align*}
0&=\Delta\phi+2\langle X,\nabla\phi\rangle-Q^-\phi+|X|^2\phi+\phi\div X\\
&=\Delta\phi+2\langle X^\perp,\nabla\phi\rangle+\Big(2|J|+\frac{1}{2}|\chi_t^-|^2\Big)\phi+(|X|^2-|X^\eta|^2)\phi+\phi\div X,
\end{align*}
since $\langle X^\eta,\nabla\phi\rangle=0$. Using that $X^\perp=\nabla\ln\phi$, $|X|^2=|X^\eta|^2+|X^\perp|^2$, and $\div X^\eta=0$, we get
\begin{align}\label{eq:aux5}
0=\Delta\phi+\frac{|\nabla\phi|^2}{\phi}+\Big(|J|+\frac{1}{4}|\chi_t^-|^2\Big)\phi
\end{align}
on $\Sigma_t$ for each $t$.

Integrating \eqref{eq:aux5} over $\Sigma_t$ gives that $\chi_t^-\equiv0$, $J\equiv 0$, and $\nabla\phi\equiv0$ on $\Sigma_t$. In particular, $\phi=\phi(t)$ depends only on $t$. Furthermore, $\chi_t^+\equiv0$ and $\chi_t^-\equiv0$ imply that $A_t\equiv0$ and $K|_{\Sigma_t}\equiv0$, where $A_t$ is the second fundamental form of~$\Sigma_t$.

Finally, we have the following conclusions:
\begin{itemize}
\item $A_t\equiv0$ for every $t\in[0,\varepsilon)$ implies that $g_t$ does not depend on $t$. Moreover, making a change of variable if necessary, we may assume that $\phi\equiv 1$. Therefore,
\begin{align*}
g=dt^2+g_\Sigma\quad\text{on}\quad U,
\end{align*}
where $g_\Sigma$ -- the induced metric on $\Sigma$ -- has Gaussian curvature
\begin{align*}
\kappa_\Sigma=c+|X^\eta|^2.
\end{align*}
\item Using the coordinate system \eqref{eq:coord_sys} on $\Sigma\cong\{0\}\times\Sigma$, we have
\begin{align*}
g=dt^2+d\theta^2+\rho^2(\theta)\,d\phi^2\quad\text{on}\quad U.
\end{align*}
Since $\eta$ is tangent to the $t$-slices, we can write
\begin{align*}
\eta=a(t,\theta,\phi)\,\p_\theta+b(t,\theta,\phi)\,\p_\phi,
\end{align*}
with $a(0,\theta,\phi)=0$ and $b(0,\theta,\phi)=1$, because $\eta=\p_\phi$ on $\Sigma$.

On the other hand, it is not difficult to see that
\begin{align*}
(\L_\eta g)_{t\theta}=\p_ta\quad\text{and}\quad(\L_\eta g)_{t\phi}=(\p_tb)\rho^2.
\end{align*}
Thus, because $\L_\eta g=0$, we have $\p_ta=\p_tb=0$, i.e.\ $a$ and $b$ does not depend on $t$. This implies that $a=a(0,\theta,\phi)=0$ and $b=b(0,\theta,\phi)=1$, that is, $\eta=\p_\phi$ on $U$.
\item We know that $X-X^\eta=X^\perp=\nabla\ln\phi\equiv 0$, that is, $X=X^\eta$, which is equivalent to $K(\p_t,\p_\theta)=\langle X,\p_\theta\rangle=0$. Moreover, $K|_{\Sigma_t}\equiv0$ for each $t$. Therefore,
\begin{align*}
K=\alpha(t,\theta,\phi)\,dt\otimes dt+\beta(t,\theta,\phi)\,(dt\otimes d\phi+d\phi\otimes dt).
\end{align*} 

On the other hand, 
\begin{align*}
\partial_\phi\alpha=(\L_{\p_\phi} K)_{tt}=0\quad\text{and}\quad\p_\phi\beta=(\L_{\p_\phi} K)_{t\phi}=0.
\end{align*}
Thus $\alpha=\alpha(t,\theta)$ and $\beta=\beta(t,\theta)$ do not depend on $\phi$. Now observe that, in this coordinate system,
\begin{align*}
|X^\eta|^2=\frac{K(\p_t,\p_\phi)^2}{\langle\p_\phi,\p_\phi\rangle}=\frac{\beta^2}{\rho^2}.
\end{align*}
Furthermore, the Gaussian curvature of $g_\Sigma=d\theta^2+\rho^2(\theta)\,d\phi^2$ is given by 
\begin{align*}
\kappa_\Sigma=-\frac{\rho''}{\rho}.
\end{align*}
Therefore, since $\kappa_{\Sigma_t}=c+|X^\eta|^2$ for each $t$, we obtain that
\begin{align*}
-\frac{\rho''(\theta)}{\rho(\theta)}=c+\frac{\beta^2(t,\theta)}{\rho^2(\theta)},
\end{align*}
or, equivalently,
\begin{align*}
-\rho''(\theta)\rho(\theta)-c\,\rho^2(\theta)=\beta^2(t,\theta).
\end{align*}
Since the left-hand side of the last equation does not depend on $t$, the same holds for $\beta=\beta(\theta)$.

Finally, we know that 
\begin{align*}
\div K-\d\tr K=\div(K-(\tr K)g)=J=0.
\end{align*}
Moreover, straightforward computations give that $(\div K)(\p_\theta)=0$. Thus
\begin{align*}
\p_\theta\alpha=\d\tr K\cdot\p_\theta=(\div K)(\p_\theta)=0,
\end{align*}
that is, $\alpha=\alpha(t)$ depends only on $t$. Lemma \ref{lemma:beta} then implies that $\beta\equiv0$, and so $\kappa_\Sigma\equiv c$.  This then implies that  $X^{\eta} = 0$ and $\omega = 0$.
\end{itemize}
\end{proof}

\section{Nariai spacetime with rotation}\label{sec:Nariai}

In this section, we consider some aspects of the rotating Nariai spacetime pertaining to some of our earlier results. This spacetime is obtained as a certain limit of the Kerr-de Sitter spacetime, in which the event and cosmological horizons coalesce. It is a vacuum solution to the Einstein equations with a positive cosmological constant $\Lambda$ that depends on a rotation parameter $a$. By setting $a=0$, one obtains the standard Nariai solution, which is the product of the two-dimensional de Sitter space and the two-sphere. For further details, see in particular \cite{AA} and \cite{AH}, from which we will be referencing a number of formulas.

The rotating Nariai spacetime has topology $\R \times S^1 \times S^2$, with metric in global coordinates given by \cite[eq.~(2.21)]{AA},
$$
\bar{g} = \Gamma(\th)(-d\tau^2 + \cosh^2\tau d\psi^2 + \alpha(\th)d\th^2)
+ \gamma(\th) (d\phi^2 - k \sinh\tau d\psi)^2,
$$
where $\tau \in (-\infty,\infty)$, $\psi \in [0, 2\pi]$ ($0 \sim 2\pi$), and $\Gamma(\th)$, $\alpha(\th)$, 
$\gamma(\th)$, and $k$ are defined in \cite[eq.~(2.19)]{AA}.\footnote{There is a small typo in the expression for 
$\gamma(\th)$, with the correct version given in \cite{AH}.}

Now consider the coordinate sphere $\Sigma: \psi = \pi$ in the spacelike hypersurface $M: \tau =0$. The area element $dA$ of $\Sigma$ is given by
$$
dA = \sqrt{\Gamma(\th)\alpha(\th)\gamma(\th)}\,d\phi\,d\th.
$$

\smallskip
Substituting the expressions for $\Gamma(\th)$, $\alpha(\th)$, and $\gamma(\th)$, we obtain
\begin{align}\label{eq:dA}
dA = \frac{r_c^2 + a^2}{1 +\frac{\Lambda a^2}{3}}\sin\th\,d\phi\,d\th,
\quad 0 \le \phi \le 2\pi,\quad 0 \le \th \le \pi,
\end{align}
and hence
\begin{align}\label{eq:area}
A(\Sigma)=4\pi\frac{r_c^2 + a^2}{1+\frac{\Lambda a^2}{3}}.
\end{align}

As part of the limiting process to obtain the rotating Nariai spacetime, $r_c$ must be a double root of the equation
$$
\Delta_r  = 0,
$$
where $\Delta_r$ is the polynomial expression in \cite[eq.~(2.3)]{AA}. It is obtained by simultaneously solving
$$
\Delta_r  = 0 \quad \text{and} \quad  \Delta_r'  = 0,
$$
which results in
\begin{align}\label{eq:rc}
{r_c^2  = \frac{3/\Lambda - a^2 + \sqrt{(3/\Lambda - a^2)^2 - 12 a^2 \cdot 3/\Lambda}}{6}}.
\end{align}

Note that, in the limit $a \to 0$, $r_c^2$ becomes
$$
r_c^2 = \frac{1}{\Lambda},
$$
and hence 
$$
A(\Sigma) = \frac{4\pi}{\Lambda} ,
$$
as is consistent with the discussion of the standard Nariai spacetime in \cite{GM2018,GM2025}.  For $a \ne 0$, one can see that 
$A(\Sigma) < 4\pi/\Lambda$. However, we would like to compare 
$A(\Sigma)$ and
$4\pi/(\Lambda + \omega)$, the area bound in \eqref{eq:main} with $c = \Lambda$.

Note that \eqref{eq:dA} and \eqref{eq:area} imply 
$$
dA = \frac{|\Sigma|}{4\pi}\sin\th\,d\phi\,d\th.
$$ 
Using this in \eqref{eq:omega}, gives
\begin{align}\label{eq:omega2}
\omega = \frac{1}{4\pi} \int_{\Sigma} |X^{\p_{\phi}}|^2 d\Omega,  \quad d\Omega = \sin\th\,d\phi\,d\theta.
\end{align}

The integrand above is given by
\[
|X^{\p_{\phi}}|^2 = \frac{[K(\nu,\p_{\phi})]^2}{\langle\p_{\phi},\p_{\phi}\rangle},
\]
where $\nu = \frac{1}{\sqrt{\Gamma(\th)}}\p_{\psi}$ is the outward unit normal to 
$\Sigma$ in $M$. Now, 
\begin{align*}
K(\nu, \p_{\phi}) &= \frac{1}{\sqrt{\Gamma(\th)}}\,K(\p_{\psi},\p_{\phi}) \\
&= \frac{1}{\sqrt{\Gamma(\th)}}\,\frac{1}{2\sqrt{\Gamma(\th)}}\,\p_{\tau}
\langle\p_{\psi},\p_{\phi}\rangle\big|_{\tau = 0}\\
&= \frac{1}{2\Gamma(\th)}\, \p_{\tau} (-\gamma(\th) k\sinh\tau)\big|_{\tau = 0}= -\frac{\gamma(\th) k}{2\Gamma(\th)}.
\end{align*}

Using $\langle\p_{\phi},\p_{\phi}\rangle = \gamma(\th)$, we arrive at
$$
|X^{\p_{\phi}}|^2 = \frac{\gamma(\th)}{4}\bigg(\frac{k}{\Gamma(\th)}\bigg)^2  
=\frac{a^2\sin^2\theta}{(r_c^2+a^2\cos^2\theta)^3}\bigg(r_c^2+\frac{\Lambda a^2r_c^2}{3}\cos^2\theta\bigg),
$$
where we have again used \cite[eq.~(2.19)]{AA}. Substituting this into \eqref{eq:omega2}, after a small manipulation we obtain
\begin{align}\label{eq:omega3}
\omega=\frac{a^2}{2r_c^4}\int_0^\pi \frac{1}{\big(1+\frac{a^2}{r_c^2}\cos^2\theta \big)^3}
\Big(1+\frac{a^2}{\ell^2}\cos^2\theta\Big)\sin^3\theta\,d\theta,
\end{align}
where $\ell$ is the `scale factor' defined by $\ell^2 = 3/\Lambda$.

While it is difficult to exactly compare $A(\Sigma)$ and $4\pi/(\Lambda + \omega)$,
we will estimate them in terms of $\e = a/\ell$. From \eqref{eq:rc}, we have
\begin{align}\label{eq:rc2}
r_c^2  &= \frac{\ell^2 - a^2 + \sqrt{(\ell^2 - a^2)^2 - 12 a^2 \cdot \ell^2}}{6} \nonumber \\
&= \frac{1 - \e^2 + \sqrt{(1-\e^2)^2 -12 \e^2}}{6} \cdot \ell^2.
\end{align}
Expanding in terms of $\e$,
\begin{align*} 
r_c^2 = \frac{\ell^2}{3}(1- 4\e^2 + 12 \e^4) + O(\e^6).
\end{align*}

Henceforth, we only retain powers up to $\e^2$.  The above then becomes
\begin{align}\label{eq:rcapprox}
r_c^2 \approx \frac{\ell^2}{3}(1- 4\e^2),
\end{align}
which then gives
\begin{align}\label{eq:arcapprox}
\frac{a^2}{r_c^2} &\approx \frac{3}{1-4\e^2} \cdot \frac{a^2}{\ell^2} = \frac{3\e^2}{1-4\e^2}  \approx 3\e^2(1 + 4\e^2) \approx 3\e^2.
\end{align}

Using this estimate, the integral in \eqref{eq:omega3} becomes
\begin{align*}
\int_0^\pi\frac{1}{(1+3\e^2\cos^2\theta)^3}
&\big(1+\e^2\cos^2\theta\big)\sin^3\theta\,d\theta\\
&\approx
\int_0^\pi\big(1- 9\e^2 \cos^2\th\big)\big(1+\e^2\cos^2\theta \big)\sin^3\theta\,d\theta  \\
&\approx \int_0^\pi\big(1- 8\e^2 \cos^2\th\big)\sin^3\theta\,d\theta \\
&= \int_0^\pi\big(\sin^3\theta- 8\e^2 \cos^2\th \sin^3\theta\big)\,d\theta \\
&= \frac43 - 8\cdot\frac{4}{15}\e^2
= \frac43\Big(1 -\frac85 \e^2 \Big).
\end{align*}

Hence,
\begin{align*}
\omega \approx \frac23\cdot\frac{a^2}{r_c^4}\Big(1 -\frac85 \e^2 \Big),
\end{align*}
and so, using the approximations \eqref{eq:rcapprox} and \eqref{eq:arcapprox},
\begin{align*}
\frac{4\pi}{\Lambda + \omega} &\approx \frac{4\pi}{\frac{3}{\ell^2} + \frac23\cdot\frac{a^2}{r_c^4}\big(1 -\frac85 \e^2 \big)} \\
&=\frac{4\pi r_c^2}{3\frac{r_c^2}{\ell^2} + \frac23\cdot\frac{a^2}{r_c^2}\big(1 -\frac85 \e^2 \big) } \\
&\approx \frac{4\pi r_c^2}{1 - 4\e^2 +2\e^2\big(1 -\frac85 \e^2\big)} 
\approx \frac{4\pi r_c^2}{1 - 2\e^2} \\
&\approx 4\pi r_c^2 \, (1 + 2\e^2) = 4\pi r_c^2\,\Big(1 + {\frac{2a^2}{\ell^2}}\Big).
\end{align*}

In a similar, but simpler, fashion, we can estimate $A(\Sigma)$. We have from \eqref{eq:area} that
\begin{align*}
A(\Sigma) &= 4\pi r_c^2\,\Bigg(\frac{1 + \frac{a^2}{r_c^2}}{1 + \frac{a^2}{\ell^2}}\Bigg)\\
&\approx 4\pi r_c^2\,\Big(\frac{1+ 3\e^2}{1+\e^2}\Big)\\
&\approx 4\pi r_c^2\,(1+3\e^2)(1 - \e^2)\\
&\approx 4\pi r_c^2\,(1+2\e^2)=4\pi r_c^2\,\Big(1 + {\frac{2a^2}{\ell^2}}\Big).
\end{align*}

Thus, there is agreement up to order $\e^2 = a^2/\ell^2$.  However, a more detailed analysis shows that a difference occurs at order $\e^4$.  $\Sigma \subset M$ does not quite saturate the area bound \eqref{eq:main}, unless $a = 0$.    

Numerical calculations illustrate how close the bound is. For example:
\begin{itemize}
\item $a = .2\,\ell$: $A(\Sigma) \approx 3.7548\,\ell^2$, $4\pi/(\Lambda + \omega) \approx 3.7651\,\ell^2$; 
\item $a = .25\,\ell$:  $A(\Sigma) \approx 3.2949\,\ell^2$, $4\pi/(\Lambda + \omega) \approx 3.3273\, \ell^2$. 
\end{itemize}

We note, from \eqref{eq:rc2}, that $a$ cannot exceed $.27\,\ell$. Also, comparing to the nonrotating case $a = 0$, $\omega = 0$, where $A(\Sigma) = 4\pi/\Lambda = 4 \pi \ell^2 /3\approx 4.1888\,\ell^2$, we see the influence of 
$\omega$ (and hence the angular momentum) on the size of $\Sigma$.

As  follows from results in \cite{GM2018,GM2025}, and the discussions there on the standard Nariai ($a = 0$) spacetime, there are initial data sets in Nariai that satisfy the assumptions of Theorem~\ref{thm:main}, in particular having surfaces that saturate the area bound.  As such, these initial data sets of course exhibit all the properties in the conclusions of Theorem~\ref{thm:main}.   

While the initial data set considered here in the rotating ($a \ne 0$) Nariai spacetime -- namely $M: \tau =0$, $\Sigma: \psi = \pi$ -- does not satisfy all of the assumptions of Theorem~\ref{thm:main}, it still exhibits  some of the rigidity properties in the conclusions of Theorem~\ref{thm:main}. For instance, as can be shown by direct computation, $M$ is foliated by $2$-spheres $\psi = \psi_0$, $ 0 \le \psi_0 < 2\pi$, each of which is a MOTS, in fact which has vanishing outward null second fundamental, $\chi^+ = 0$. This behavior can be understood from a more general perspective.  

The rotating Nariai spacetime contains stationary (time-independent) regions, bounded by cosmological horizons.  The metric in stationary coordinates is given by  \cite[eq.~(2.19)]{AH},
$$
\bar{g} = \Gamma(\th)\Big(-(1-r^2)dt^2 + \frac{1}{1-r^2} dr^2 + \alpha(\th)d\th^2\Big)
+ \gamma(\th) (d\phi^2 - k \sinh\tau d\psi)^2.
$$
Under a suitable change of coordinates, one verifies that the slice $t=0$ corresponds to our initial data manifold $M: \tau = 0$. We then have the following proposition applicable to the situation at hand.  

\begin{proposition}
Let $(M,g,K)$ be an initial data set in a spacetime $(\bar{M},\bar{g})$  satisfying the null energy condition (NEC), $\textrm{Ric}(X,X) \ge 0$ for all null vectors $X$. Let $\Sigma$  be a locally weakly outermost MOTS in $M$ contained in a stationary region of $\bar{M}$.  
Then there exists an  outer neighborhood  $U \cong [0, \e) \times \Sigma$ of 
$\Sigma$ in $M$ foliated by surfaces $\Sigma_t \cong \{t\} \times \Sigma$, $t \in [0, \e)$, with $\chi^+(t) = 0$.
\end{proposition} 

\begin{proof}[Sketch of the proof]
By stationarity, there exists a timelike Killing vector field $T$ defined in a neighborhood of $\Sigma$.  Move $\Sigma$ under the flow of $T$ a small parameter time $t$ to the past
to obtain a MOTS $S_t$ with future directed outward null normal field $\ell_t^+$.  For $a > 0$ small, the null normal exponential map 
$\Phi : [0,a) \times \Sigma$, 
$\Phi(s, x) = \exp_x(s\ell_t^+)$  is an embedding, and its image is a null hypersurface $\H$, with associated null vector field $N= \Phi_*(\p/\p s)$. 
By keeping things sufficiently local,  $\H$ will intersect $M$ in a surface 
$\Sigma_t$ homologous to $\Sigma$.  Since $S_t$ has zero outward null expansion, $\theta_{S_t}^+ = 0$, the Raychaudhuri equation, together with the NEC, implies that $\Sigma_t$ has null expansion 
$\theta_{\Sigma_t}^+ \le 0$.  In fact, one must have $\theta_{\Sigma_t}^+ = 0$, as otherwise, by \cite[Lemma~5.2]{AM},
one could deform $\Sigma_t$ outward in $M$ to an outer trapped surface, contradicting the weakly outermost condition.  But now, by the Raychaudhuri equation, this forces the full null second fundamental form of $\H$ to vanish, at least up to $M$.  Since $\Sigma_t \subset \H$, this in particular implies that 
$\chi_t^+ = 0$. As this hold for all sufficiently small $t$, the result follows.
\end{proof}

\begin{remark}
Certain elements of the proof are similar to the proof of Theorem~3 in \cite{CM}.  Also, there are a couple of ways to see that the spheres $\psi = \psi_0$ in our rotating Nariai example are locally weakly outer trapped, for example, by using the future null geodesic completeness of rotating Nariai and a one-sided version of the Penrose singularity theorem applied to the universal cover; see also \cite{CGL} for a somewhat different argument.  
\end{remark}

\section*{Acknowledgments}

We are very grateful to Dionysios Anninos for valuable input regarding the rotating Nariai spacetime. We also thank Marcus Khuri for comments that helped to motivate this work. The work of the first-named author was partially supported by the University of Miami. The work of the second-named author was partially supported by the Conselho Nacional de Desenvolvimento Científico e Tecnológico (CNPq, Grant No.~309867/2023-1) and by the Coordenação de Aperfeiçoamento de Pessoal de Nível Superior (CAPES/MATH-AMSUD, Grant No.~88887.985521/2024-00), Brazil.

\section*{Data availability}

No data were generated or analyzed in this study.

\newpage

\bibliographystyle{plain}
\bibliography{bibliography.bib}

\end{document}